\documentclass{article}
\usepackage{amsfonts,amsmath,amssymb,amsthm,cite,enumerate,graphicx,latexsym,mathrsfs}
\usepackage{authblk}
\usepackage{hyperref}
\usepackage{float}
\makeatletter

\newcommand{\Rmnum}[1]{\expandafter\@slowromancap\romannumeral #1@}
\makeatother

\numberwithin{equation}{section}

\newtheorem{lemma}{Lemma}[section]
\newtheorem{proposition}{Proposition}[section]
\newtheorem{remark}{Remark}[section]

\newtheorem{theorem}{Theorem}[section]

\newcommand\theref[1]{Theorem~\ref{#1}}
\newcommand\lemref[1]{Lemma~\ref{#1}}

\newcommand\proref[1]{Proposition~\ref{#1}}

\allowdisplaybreaks

\allowdisplaybreaks[4]

\title{Time-periodic solution to nonhomogeneous isentropic compressible Euler equations with time-periodic boundary conditions}
\author{Huimin Yu\thanks{{e}-mail: hmyu@sdnu.edu.cn} \quad Xiaomin Zhang\thanks{{e}-mail: 1340473344@qq.com} \quad Jiawei Sun\thanks{Corresponding author {e}-mail: sunjiawei0122@163.com}
 \\ \small\textit{ Department of mathematics, Shandong Normal University, Jinan 250014 China}}
\begin{document}
\date{}
\maketitle
\begin{center}
\begin{minipage}{130mm}{\small
\textbf{Abstract}:
In this paper, we study one-dimensional nonhomogeneous isentropic compressible Euler equations with time-periodic boundary conditions. With the aid of the energy methods, we prove the existence and uniqueness of the time-periodic supersonic solutions after some certain time.      


}      

\end{minipage}
\end{center}
\section{Introduction}
\indent\indent In this paper, we study the initial-boundary value problem for one dimensional isentropic compressible Euler system with nonhomogeneous term
\begin{equation}\label{a1}
\left\{\begin{aligned}
&\rho_{t}+(\rho u)_{x}=0,\\
&(\rho u)_{t}+(\rho u^{2}+p)_{x}=\alpha\rho u,\\
&\rho(0,x)=\underline{\rho},\, u(0,x)=\underline{u},\\
&\rho(t,0)=\rho_{l}(t),\, u(t,0)=u_{l}(t)
\end{aligned}\right.
\end{equation}
 in the domain $(t,x)\in[0,\infty)\times[0,L]$ with a given constant $L>0$, where $\rho\geq 0$ and $u\in \mathbb{R}$ represent the density and velocity of the fluid, respectively. $p$ presents the pressure function, which is given by
  \begin{equation}\label{p}
  p=a\rho^\gamma,~~a>0,~~\gamma>1
 \end{equation}
  for the isentropic polytropic gas. Here the initial data $\underline{\rho}>0$ and $\underline{u}>0$ are two constants, the boundary conditions $\rho_{l}(t),~u_{l}(t)$ are periodic functions with a period $P>0$ and satisfy the corresponding compatibility conditions, i.e.
 \begin{equation}\label{bp}
 \rho_{l}(t)=\rho_{l}(t+P), ~~u_{l}(t)=u_{l}(t+P),
 \end{equation}
 and
  \begin{equation}\label{CC}
 \rho_{l}(0)=\underline{\rho}, ~~u_{l}(0)=\underline{u}.
 \end{equation}
 The term $\alpha\rho u$ in the second equation of (\ref{a1}) is the so-called damping (acceleration) effect on the
fluid for $\alpha<0$ ($\alpha>0$). Here $\alpha$ denote the external force coefficient.

We consider the inflow problem for the one-dimensional isentropic compressible Euler equation with source term $\alpha(t)\rho u$. We focus on the conditions, imposed on the external force coefficient $\alpha(t)$, that the Euler equation can trigger a time periodic solution by a time periodic boundary condition $(\rho_l(t), u_l(t))$. We consider this problem for two reasons:
  In one hand, from a physical point of view, the force produced by the wall of (constant-area, convergent, or divergent) duct \cite{Shapiro} can be regarded as some source term added to the Euler system. On the other hand, Yuan in \cite{Yuan} considered a similar problem for Euler equation without source term, we want to extend the results to a more general model.\\
   \indent Moreover, since the initial data is a constant which is independent of space variable $x$, the boundary data $\rho(0,t), u (0,t)$ is only functions of time. A natural consideration is whether the system have a stable smooth solution, which is also independent of space $x$, for later time?

There are many important progresses on the studies of the time-periodic solutions of the partial differential equations, for instance the viscous fluids equations~\cite{Cai,Ma,Jin,Luo,Mats} and the hyperbolic conservation laws~\cite{G,Tak,Tem,Oh}. All of the works mentioned above discuss the time-periodic solutions which are driven by the time-periodic external forces or the piston motion. For the case of the time-periodic boundary condition, there are fewer works on the time-periodic solutions of the hyperbolic conservation laws. Yuan considered the time-periodic solution for the isentropic compressible Euler equations (i.e. $\alpha=0$) triggered by time-periodic supersonic boundary condition in~\cite{Yuan}. For the quasilinear hyperbolic system
\begin{align*}
\partial_{t}u+A(u)\partial_{x}u=0,\quad (t,x)\in {\mathbb{R}}\times [0,L]
\end{align*}
with a more general time-periodic boundary conditions, Qu studied the existence and stability of the time-periodic solutions around a small neighborhood of $u\equiv0$ in~\cite{Qu}. Intensive literatures have investigated the
isentropic compressible Euler system with source terms. We refer to~\cite{Chen,Sui} and the references therein for the results on the formation of singularity, \cite{Yu,Cao,Fang,Huang,H} for the existence and large time behavior of weak solutions, and~\cite{Chen,Cui,Li,Hou,Hsiao,Mar,Nishihara} for
the asymptotic behavior of smooth solutions, etc.\\
\indent We assume the coefficient $\alpha=\alpha(t)=\alpha(t+P)$, which belongs to $C^{2}$ and satisfies
 \begin{align}
 &0\leq \int_{0}^{t}\alpha(s)ds<+\infty,\quad \forall t\in[0,\infty), \label{a2}\\
 &\int_{0}^{P}\alpha(t)dt=0.\label{a3}
 \end{align}
Obviously, $\alpha(t)\equiv0$ satisfy $(\ref{a2})$ and $(\ref{a3})$. In this case, the equations $(\ref{a1})$ turns into the usual isentropic compressible Euler equations. Yuan proved the existence of the time-periodic supersonic solution derived by the periodic boundary condition $(\ref{a1})_4$ in~\cite{Yuan}, where the initial data is assumed near a supersonic constant state $(\underline{\rho},\underline{u})$ with
 \begin{align}
 \underline{u}>\underline{c}:=\sqrt{a\gamma}\underline{\rho}^{\frac{\gamma-1}{2}}. \label{snc}
 \end{align}
 Moreover, the existence of time-periodic weak solution with small initial data around $(\underline{\rho},\underline{u})$ is also considered in~\cite{Yuan}.\\
 \indent In this paper, we consider the one dimensional isentropic compressible Euler equations with the time-periodic external force coefficient $\alpha(t)$, and prove the existence of a time-periodic supersonic solution induced by the time-periodic boundary condition. Unlike the constant supersonic state $(\underline{\rho},\underline{u})$, here we consider the time-periodic background supersonic solution $(\underline{\rho},e^{\int_{0}^{t}\alpha(s)ds}\underline{u})$ and prove that the problem ~$(\ref{a1})$-$(\ref{a3})$ has a time-periodic supersonic solution after some certain time. Our precise results are sated below.
\noindent\begin{theorem}\label{T1}
There exist positive constants $\varepsilon_{0},T_{0}$ and $C_{0}$, such that if
\begin{align}\label{a4}
\|\rho_{l}-\underline{\rho}\|_{H^{2}([0,P])}+\|u_{l}-e^{\int_{0}^{t}\alpha(s)ds}\underline{u}\|_{H^{2}[0,P]}\leq\varepsilon
\end{align}
for any $\varepsilon\leq\varepsilon_{0}$, then the initial-boundary value problem~\eqref{a1} admits a unique solution$(\rho,u)\in C^{1}([0,+\infty)\times [0,L])$, which satisfies
\begin{align}
&\rho(t+P,x)=\rho(t,x), u(t+P,x)=u(t,x),\quad \forall t>T_{0},x\in[0,L], \label{PC1}\\
&\|\rho-\underline{\rho}\|_{C^{1}([0,\infty)\times[0,L])}+\|u-e^{\int_{0}^{t}\alpha(s)ds}\underline{u}\|_{C^{1}([0,\infty)\times[0,L])}\leq C_{0}\varepsilon. \label{a5}
\end{align}
\end{theorem}
\indent There are a few remarks in order.
\noindent\begin{remark}\label{R1}
The conditions (\ref{a2}) and (\ref{a3}), imposed on the external force coefficient $\alpha (t)$, are proposed to ensure the background solution $(\underline{\rho},e^{\int_{0}^{t}\alpha(s)ds}\underline{u})$ is a periodic supersonic solution.

As for the assumption on $\alpha(t)$ in~\eqref{a2}, namely
$$
\int_0^t \alpha(s)ds> 0,
$$
can be extended to
$$
\int_0^t \alpha(s)ds> \ln \underline{c} - \ln \underline{u},
$$
which is imposed to guarantee the background solution is supersonic everywhere.
\end{remark}
\noindent\begin{remark}\label{R2}
The conditions (\ref{a2}) and (\ref{a3}) are only sufficient conditions to ensure the existence of time-periodic solution. To see this, we assume $\alpha\equiv const.\neq0$ which do not satisfy these two conditions. However, the periodic boundary condition can trigger time-periodic solution, too. We will deal with this problem in the following works.
\end{remark}
\noindent\begin{remark}\label{R3}
Supersonic is an essential assumption which implies that all characteristics propagate forward in both space and time. In particular, this implies that the effects of any nonlinear interaction
at $(x_0, t_0)$ are confined to the region ${x > x_0, t > t_0}$. Supersonic condition plays a very important role in our proof. For example, in the case of exchanging $x$ and $t$ in (\ref{b2}), we need $\lambda_{1}$ and $\lambda_{2}$ are not zero in all $(t,x)\in[0,+\infty)\times[0,L]$; in the proof of~\theref{thm2}, the positive definite matrix $\Lambda$ is significant. However, these two things can be ensured by the supersonic condition.
\end{remark}
\noindent\begin{remark}\label{R4}
The results can be extend to higher estimate by using a standard $H^s$ energy estimates together with a finite speed of propagation/domain of dependence
argument, here we omit the detail.
\end{remark}
\noindent\begin{remark}\label{R5}
Using a similar method, we can also consider the problem with nonlinear friction $\alpha(t)\rho u |u|^\theta$, for $\theta>0$, here omit the detail.
\end{remark}
The rest of the paper is organized as follows. In Section 2, we give some basic facts for the Euler equations. In Section 3, we show that the smooth solution to the initial-value problem~\eqref{a1} is time-periodic after some certain time when it is a small perturbation of the supersonic background state. In Section 4, we give the proof of~\theref{T1} by the aid of two solutions of the two initial(-boundary) value problems stated in~\theref{thm1} and~\theref{thm2}.
\section{Preliminary and Formulation}\label{s2}
\indent\indent We first introduce some basic facts for system $(\ref{a1})$. The eigenvalues are
\begin{align}
\lambda_{1}=u-c=u-\sqrt{a \gamma}\rho^{{\gamma-1}\over 2},\quad \lambda_{2}=u+c=u+\sqrt{a \gamma}\rho^{{\gamma-1}\over 2}
\end{align}
and the corresponding right eigenvectors are \begin{align*}
\vec{r}_1=\frac{1}{\sqrt{\rho^{2}+c^{2}}}(\rho, -c)^T,\quad \vec{r}_2=\frac{1}{\sqrt{\rho^{2}+c^{2}}}(\rho, c)^T.
\end{align*}
With the aid of the Riemann invariants
\begin{align}\label{b1}
r=\frac{1}{2}(u-\frac{2}{\gamma-1}c),\quad s=\frac{1}{2}(u+\frac{2}{\gamma-1}c),
\end{align}
the problem~\eqref{a1} can be rewritten as follows
\begin{equation} \label{IVP1}
\begin{cases}
&r_{t}+\lambda_{1}(r,s)r_{x}=\frac{\alpha(t)(r+s)}{2},\\
&s_{t}+\lambda_{2}(r,s)s_{x}=\frac{\alpha(t)(r+s)}{2},\\
&r(0,x)=\underline{r},\, s(0,x)=\underline{s},\\
&r(t,0)=r_{l}(t),\, s(t,0)=s_{l}(t)
\end{cases}
\end{equation}
with
\begin{align}
\underline{r}&=\frac{1}{2}\underline{u}-\frac{\sqrt{a\gamma}}{\gamma-1}\underline{\rho}^{\frac{\gamma-1}{2}},\quad
\underline{s}=\frac{1}{2}\underline{u}+\frac{\sqrt{a\gamma}}{\gamma-1}\underline{\rho}^{\frac{\gamma-1}{2}},\notag\\
r_{l}(t)&=\frac{1}{2}u_{l}(t)-\frac{\sqrt{a\gamma}}{\gamma-1}{\rho}^{\frac{\gamma-1}{2}}_{l}(t),\quad
s_{l}(t)=\frac{1}{2}u_{l}(t)+\frac{\sqrt{a\gamma}}{\gamma-1}{\rho}^{\frac{\gamma-1}{2}}_{l}(t). \label{pb1}
\end{align}
After exchanging the roles of $t$ and $x$ in~\eqref{IVP1}, we consider the following Cauchy problem
\begin{equation}\label{b2}
\left\{\begin{array}{lll}
r_{x}+\frac{1}{\lambda_{1}}r_{t}=\frac{\alpha(t)(r+s)}{2\lambda_{1}},\\
 \\
s_{x}+\frac{1}{\lambda_{2}}s_{t}=\frac{\alpha(t)(r+s)}{2\lambda_{2}},\\
r(t,0)=r_{\ast}(t)=\left\{
\begin{array}{lr}
r_{l}(t),~~~t>0,\\
 \\
r_{\alpha}(t),~~~t\leq0,
\end{array}\right.\\
 \\
s(t,0)=s_{\ast}(t)=\left\{
\begin{array}{lr}
s_{l}(t),~~~t>0,\\
\\
s_{\alpha}(t),~~~t\leq0.
\end{array}\right.
\end{array}\right.
\end{equation}
Here
\begin{align}
r_{\alpha}(t)=\frac{1}{2}e^{\int_{0}^{t}\alpha(s)ds}\underline{u}-\frac{1}{\gamma-1}\underline{c},\quad
s_{\alpha}(t)=\frac{1}{2}e^{\int_{0}^{t}\alpha(s)ds}\underline{u}+\frac{1}{\gamma-1}\underline{c},
\end{align}
which can be looked as a background periodic solution of problem~\eqref{b2} satisfying
\begin{align*}
r'_{\alpha}(t)=\frac{\alpha(t)}{2}(r_\alpha+s_\alpha),\quad
s'_{\alpha}(t)=\frac{\alpha(t)}{2}(r_\alpha+s_\alpha).
\end{align*}

\section{Periodic Solution}\label{s3}
\indent\indent In this Section, we will prove the supersonic smooth solution $(r,s)(t,x)$, satisfying the a-priori estimate
\begin{align}\label{b3}
\|r-r_{\alpha}\|_{C^{1}(R\times[0,L])}+\|s-s_{\alpha}\|_{C^{1}(R\times[0,L])}\leq \varepsilon
\end{align}
for some sufficiently small $\varepsilon>0$, is a time-periodic function when $t$ is big enough. To this end, we set
$$
m=\left(
 \begin{array}{c}
 r(t,x)-r_{\alpha}(t)\\
 s(t,x)-s_{\alpha}(t)\\
 \end{array}
\right),
$$
which satisfies  that
\begin{equation} \label{DU1}
\left\{
\begin{array}{llll}
&m_{x}+\Lambda(t,x) m_{t}=\frac{\alpha (t)}{2}\Lambda(t,x)
\left(
\begin{array}{c}
 r-r_{\alpha}+s-s_{\alpha}\\
 r-r_{\alpha}+s-s_{\alpha}\\
 \end{array}
 \right),\\
&m(t,0)=\left\{
\begin{aligned}
&(r_{l}(t)-r_{\alpha}(t), s_{l}(t)-s_{\alpha}(t))^{T},\quad t>0,\\
&0,~~~~~~~~~~~~~~~~~~~~~~~~~~~~~~~~~~~~\quad t\leq0,
\end{aligned}\right.
\end{array}
\right.
\end{equation}
where
\begin{align}
\Lambda=\Lambda(t,x)=\left(
            \begin{array}{cc}
            \frac{1}{\lambda_{1}(r(t,x),s(t,x))} & 0\\
            0 & \frac{1}{\lambda_{2}(r(t,x),s(t,x))}
            \end{array}
            \right).
\end{align}
From \eqref{b3}, for small $\varepsilon>0$, we deduce that the flow is still supersonic and
\begin{align}
\lambda_{0}=\inf_{t\geq0,x\in[0,L]}\lambda_{1}(r(t,x),s(t,x))>0.
\end{align}
We next prove that
\begin{equation}\label{b4}
\begin{split}
&r(t+P,x)-r_{\alpha}(t+P)=r(t,x)-r_{\alpha}(t),\\
&s(t+P,x)-s_{\alpha}(t+P)=s(t,x)-s_{\alpha}(t),
\end{split}
\end{equation}
for any $t\geq T_{0}:=\frac{L}{\lambda_{0}}$ and $x\in[0,L]$. Then we can conclude $r(t,x)$ and $s(t,x)$ are time-periodic functions.

Write
$$
V(t,x)=m(t+P,x)-m(t,x).
$$
After a straightforward computation, we have from~\eqref{DU1} that
\begin{align}\label{b5}
\left\{\begin{aligned}
&V_{x}+\Lambda(t,x)V_{t}=G(t,x),\\
&V(t,0)=\left\{
\begin{aligned}
&(r_{l}(t+P)-r_{\alpha}(t+P),s_{l}(t+P)-s_{\alpha}(t+P))^{T},\\
&~~~~~~~~~~~~~~~~~~~~~~~~~~~~~~~~~~~~~~~~~~~~~~~~~~~~~~~P\leq t\leq0,\\
&0,~~~~~~~~~~~~~~~~~~~~~~~~~~~~~~~~~~~~~~~~~~~~~~~~t>0,~ \text{or}\, ~t<-P,
\end{aligned}\right.
\end{aligned}\right.
\end{align}
where
\begin{align*}
G(t,x)=&\frac{\alpha(t)}{2}\Lambda(t+P,x)\left(
\begin{array}{c}
r(t+P,x)+s(t+P,x)\\
r(t+P,x)+s(t+P,x)\\
\end{array}
\right)\\
&-\frac{\alpha(t)}{2}\Lambda(t,x)\left(
                            \begin{array}{c}
                            r(t,x)+s(t,x)\\
                            r(t,x)+s(t,x)\\
                            \end{array}
\right)\\
&-(\Lambda(t+P,x)-\Lambda(t,x))\left(
              \begin{array}{c}
              r'_{\alpha}(t+P)\\
              s'_{\alpha}(t+P)\\
              \end{array}
\right)\\
&-\big(\Lambda(t+P,x)-\Lambda(t,x)\big)m_{t}(t+P,x).
\end{align*}
In above calculations, we have used the facts that
$$\alpha(t+P)=\alpha(t), ~{\rm{and}}~ \int_0^{t+P}\alpha(s)ds=\int_0^t\alpha(s)ds.$$

By~\eqref{a2}, \eqref{a3} and \eqref{b3} with $\varepsilon>0$ small enough, we have for any $t\in[0,\infty)$ and $x\in[0,L]$ that
\begin{align}
&|m_{t}(t+P,x)|\leq C_{2}\varepsilon,\label{NE1}\\
&|\Lambda(r(t,x),s(t,x))|+|\Lambda_{t}(r(t,x),s(t,x))|\leq C_{2},\label{b16}\\
&|r(t+P,x)+s(t+P,x)|+|r'_{\alpha}(t+P)|+|s'_{\alpha}(t+P)|\leq C_{2},\label{NE2}\\
&|\Lambda(t+P,x)-\Lambda(t,x)|\leq |\Lambda_{r}| |r(t+P,x)-r(t,x)|\notag\\
&~~~~~~~~~~~~~~~~~~~~~~~~~~~~~~~+|\Lambda_{s}| |s(t+P,x)-s(t,x)| \leq C_{2}|V(t,x)|,\label{b6}
\end{align}
where
$$
\Lambda_{r}=\left(
            \begin{array}{cc}
            -\frac{1}{\lambda_{1}^{2}}\frac{\gamma+1}{2} & 0\\
            0 & -\frac{1}{\lambda_{2}^{2}}\frac{3-\gamma}{2}\\
\end{array}
\right),\quad
\Lambda_{s}=\left(
            \begin{array}{cc}
            -\frac{1}{\lambda_{1}^{2}}\frac{\gamma-3}{2} & 0\\
            0 & -\frac{1}{\lambda_{2}^{2}}\frac{\gamma+1}{2}\\
\end{array}
\right),
$$
and $C_{2}>0$ only depends on $\underline{\rho},\underline{u},\gamma$ and $L$. Then we obtain from~\eqref{NE1}-\eqref{b6} that
\begin{equation}\label{b7}
\begin{split}
|G(t,x)|=&|\frac{\alpha(t)}{2}[\Lambda(t+P,x)-\Lambda(t,x)]\left(
            \begin{array}{c}
            r(t+P,x)+s(t+P,x)\\
            r(t+P,x)+s(t+P,x)\\
            \end{array}
\right)\\
&+\frac{\alpha(t)}{2}\Lambda(t,x)\left(
            \begin{array}{c}
            r(t+P,x)-r(t,x)+s(t+P,x)-s(t,x)\\
            r(t+P,x)-r(t,x)+s(t+P,x)-s(t,x)\\
            \end{array}
\right)\\
&-(\Lambda(t+P,x)-\Lambda(t,x))\left(
            \begin{array}{c}
            r'_{\alpha}(t+P)\\
            s'_{\alpha}(t+P)\\
            \end{array}
\right)\\
&-(\Lambda(t+P,x)-\Lambda(t,x))m_{t}(t+P,x)|\\
\leq& C_{3} |V(t,x)|,
\end{split}
\end{equation}
where $C_{3}>0$ only depends on $\underline{\rho},\underline{u},\gamma$ and $L$.

Fixing a point $(t',x')$ with $t'>T_{0}$ and $0<x'<L$, we draw the slow and fast characteristic curves $\Gamma_{1}:t=t_{1}(x)$ and $\Gamma_{2}:t=t_{2}(x)$
\begin{align*}
\frac{dt_{1}}{dx}=\frac{1}{\lambda_{1}(r(t_{1},x),s(t_{1},x))},\,\,t_{1}(x')=t',\\
\frac{dt_{2}}{dx}=\frac{1}{\lambda_{2}(r(t_{2},x),s(t_{2},x))},\,\,t_{2}(x')=t',
\end{align*}
which intersect the t-axis. Noting that $\Gamma_{1}$ is below $\Gamma_{2}$ as shown in the figure below.
\begin{figure}[H]
\centering
\includegraphics[width=8cm]{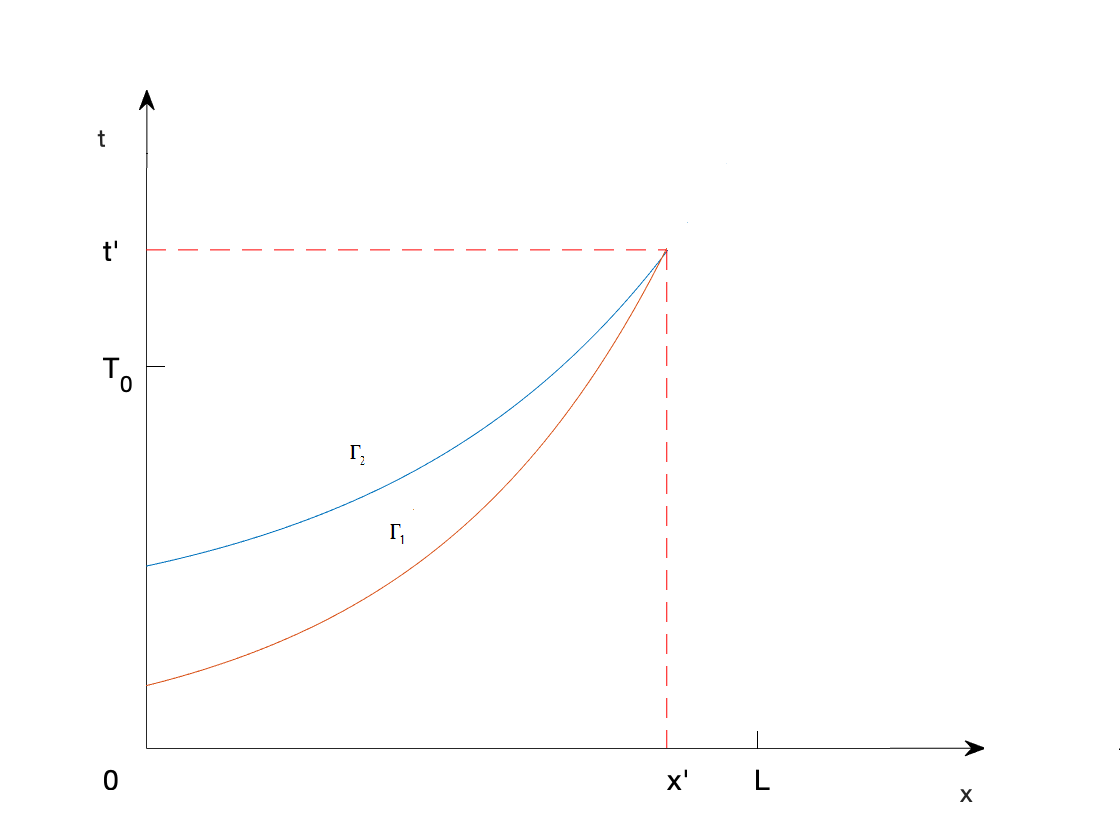}
\end{figure}
For any $x\in[0,x')$, we set
\begin{align}\label{b8}
I(x)=\frac{1}{2}\int_{t_{1}(x)}^{t_{2}(x)}|V(t,x)|^{2}dt.
\end{align}
Due to $t'>T_{0}:=\frac{L}{\lambda_{0}}$ and the definition of $\lambda_{0}$, we have $(t_{1}(0),t_{2}(0))\subset(0,+\infty)$. Then by (\ref{b5}), $V(t,0)\equiv0$, therefore
\begin{align}
I(0)=0. \label{inl1}
\end{align}
Taking derivative on $I(x)$ with regard to $x$, it holds that
\begin{align*}
I'(x)=& \int_{t_{1}(x)}^{t_{2}(x)}{V(t,x)^{T}V_{x}(t,x)}dt+\frac{1}{2}|V(t_{2}(x),x)|^{2}{\frac{1}{\lambda_{2}(t_{2}(x),x)}}\\
&-\frac{1}{2}|V(t_{1}(x),x)|^{2}{\frac{1}{\lambda_{1}(t_{1}(x),x)}}\\
\leq&-\int_{t_{1}(x)}^{t_{2}(x)}V(t,x)^{T}\Lambda(t,x)V_{t}(t,x)dt+\int_{t_{1}(x)}^{t_{2}(x)}V(t,x)^{T}G(t,x)dt\\
&+\frac{1}{2}V(t,x)^{T}\Lambda(t,x)V(t,x)|_{t=t_{1}(x)}^{t=t_{2}(x)}\\
=&-\frac{1}{2}\int_{t_{1}(x)}^{t_{2}(x)}\Big{(}(V(t,x)^{T}{\Lambda(t,x)V(t,x))}_{t}-V(t,x)^{T}\Lambda_{t}(t,x)V(t,x)\Big{)}dt\\
&+\int_{t_{1}(x)}^{t_{2}(x)}V(t,x)^{T}G(t,x)dt+\frac{1}{2}V(t,x)^{T}\Lambda(t,x)V(t,x)|_{t=t_{1}(x)}^{t=t_{2}(x)}\\
=&\frac{1}{2}\int_{t_{1}(x)}^{t_{2}(x)}V(t,x)^{T}\Lambda_{t}(t,x)V(t,x)dt+\int_{t_{1}(x)}^{t_{2}(x)}V(t,x)^{T}G(t,x)dt\\
\leq&(\frac{C_{2}}{2}+C_{3})I(x).
\end{align*}
In the last inequality above, we have used~\eqref{b16} and \eqref{b7}. Using the Gronwall's inequality, we can get from~\eqref{inl1} that $I(x)\equiv0$. Then it follows from the continuity of $I(x)$ that $I(x')=0$ which implies that $V(t',x')=0$.

Because of the arbitrary of $(t',x')$, we have
\begin{align*}
V(t,x)\equiv0,\quad \forall~t>T_{0},\,\, x\in[0,L].
\end{align*}
Thus, we prove~\eqref{b4}. Since $\alpha(t)$ and $\int_{0}^{t}\alpha(s)ds$ are periodic functions with a period $P>0$, then we get from~\eqref{b4} that $r(t)$ and $s(t)$ are also periodic functions with a period $P>0$. Namely,
\begin{align*}
r(t+P,x)=r(t,x),\quad s(t+P,x)=s(t,x),\,\, \forall~t>T_{0},\,x\in[0,L].
\end{align*}
\section{Existence of Solutions}\label{S4}
\indent\indent In this Section, we prove the existence of the time-periodic solution $m=(m_{1},m_{2})^{T}:=(r-r_{\alpha}, s-s_{\alpha})^{T}$ for $(t,x)\in[0,+\infty)\times[0,L]$ to the following Cauchy problem
\begin{equation}\label{c1}
\left\{\begin{aligned}
&m_{x}+\Lambda m_{t}=\frac{\alpha(t)}{2}\Lambda
\left(
\begin{array}{c}
 r-r_{\alpha}+s-s_{\alpha}\\
 r-r_{\alpha}+s-s_{\alpha}\\
 \end{array}
 \right),\\
&m(t,0)=m_{0}(t)=
(r_{l}(t)-r_{\alpha}(t), s_{l}(t)-s_{\alpha}(t))^{T}.
\end{aligned}\right.
\end{equation}
Firstly, for the smooth solution $m$ to the Cauchy problem~\eqref{c1} in $(t,x)\in[T_{0},T_{0}+P]\times [0,L]$, we will establish an a-priori estimate
\begin{eqnarray}\label{c3}
\sup\limits_{x\in[0,L]}\| m \|_{H^{2}([T_{0},T_{0}+P])}<\delta
\end{eqnarray}
for some constant $\delta>0$ small enough.
Denoting
$$
\left(
  \begin{array}{c}
    r-r_{\alpha}+s-s_{\alpha} \\
    r-r_{\alpha}+s-s_{\alpha} \\
  \end{array}
\right)= \left(
  \begin{array}{cc}
    1 & 1 \\
    1 & 1 \\
  \end{array}
\right)\left(
  \begin{array}{c}
    r-r_{\alpha} \\
    s-s_{\alpha} \\
  \end{array}
\right)=Am,
$$
and letting $\tilde{\Lambda}=\tilde{\Lambda}(t,x)=\Lambda A$,  then $\eqref{c1}_{1}$ has the form
\begin{equation}\label{c2}
m_{x}+\Lambda m_{t}=\frac{\alpha(t)}{2}\tilde{\Lambda}m.
\end{equation}
Without losing generality, we let the period $P=1$ for simplicity in the rest of the paper.
Now, we have the following Lemma:
\begin{lemma}\label{t1}
For sufficiently small $\delta>0$, it holds that
\begin{align}
\lambda_{1}(t,x)\geq\frac{1}{2}(\underline{u}-\underline{c}),\quad \lambda_{2}(t,x)\geq\frac{1}{2}(\underline{u}+\underline{c}) \label{UE1}
\end{align}
for any $(t,x)\in [T_{0},T_{0}+1]\times[0,L]$.
\end{lemma}
\begin{proof}
 Since
\begin{align}
|\lambda_{1}(r,s)|
 &=|\frac{\gamma+1}{2}(r-r_{\alpha})-\frac{\gamma-3}{2}(s-s_{\alpha})+u_{\alpha}-\underline{c}|\notag\\
 &\geq u_{\alpha}-\underline{c}-\frac{\gamma+1+|\gamma-3|}{2}|m|\notag\\
 &\geq \underline{u}-\underline{c}-\frac{\gamma+1+|\gamma-3|}{2}K_{1}\delta, \label{1t1}
 \end{align}
where we have used $\int_{0}^{t}\alpha(s)ds\geq0$ and the following embedding inequality
\begin{equation}\label{emd1}
\sup_{t\in[T_{0},T_{0}+1],\,x\in[0,L]}|m|\leq K_{1}\sup_{x\in[0,L]}\|m(x)\|_{H^{2}([T_{0},T_{0}+1])}<K_{1}\delta
\end{equation}
for some constant $K_{1}>0$. By choosing $\delta\leq\frac{\underline{u}-\underline{c}}{(\gamma+1+\mid\gamma-3\mid)K_{1}}$, we obtain from~\eqref{1t1} that $\lambda_{1}\geq\frac{1}{2}(\underline{u}-\underline{c})$.
The second inequality in~\eqref{UE1} can be obtained in a similar way, the details are omitted here. The proof of~\lemref{t1} is completed.
\end{proof}
\begin{proposition}\label{Pro1}
Let $m(t,x)$ be a smooth time-periodic solution to the Cauchy problem~\eqref{DU1} satisfying~\eqref{c3}. Then it holds that
\begin{align}
\frac{d}{dx}\|m\|^{2}_{H^{2}([T_{0},T_{0}+1])}\leq C\|m\|^{2}_{H^{2}([T_{0},T_{0}+1])} \label{1Pro1}
\end{align}
for some constant $C>0$.
\end{proposition}
\begin{proof}
Multiplying~\eqref{c2} by $m^{T}$ and integrating it from $T_{0}$ to $T_{0}+1$, we have
\begin{align}
\frac{d}{dx}\frac{1}{2}\|m\|_{L^{2}}^{2}&=\int_{T_{0}}^{T_{0}+1}\Big{(}-m^{T}\Lambda m_{t}+\frac{\alpha(t)}{2}m^{T}\tilde{\Lambda}m\Big{)}dt\notag\\
&=\int_{T_{0}}^{T_{0}+1}\Big{(}\frac{1}{2}m^{T}\Lambda_{t} m+\frac{\alpha(t)}{2}m^{T}\tilde{\Lambda}m\Big{)}dt\notag\\
&\leq\frac{1}{2}(\sup_{t\in[T_{0},T_{0}+1] \atop x\in[0,L]}|\Lambda_{t}|+\sup_{t\in[T_{0},T_{0}+1] \atop x\in[0,L]}|\alpha(t)||\tilde{\Lambda}|)\|m\|_{L^{2}}^{2},
 \label{2pro1}
\end{align}
where we have used the integration by parts in the second identity. From~\lemref{t1} and~\eqref{c3} we know
\begin{align}
\sup_{t\in[T_{0},T_{0}+1]\atop x\in[0,L]}|\tilde{\Lambda}|&\leq \sup\limits_{t\in[T_{0},T_{0}+1]\atop x\in[0,L]}\frac{1}{\lambda_{1}}\leq C_{2}:=2(\underline{u}-\underline{c})^{-1},\label{3Pro1}\\
\sup_{t\in[T_{0},T_{0}+1]\atop x\in[0,L]}|\Lambda_{t}|&\leq\sup\limits_{t\in[T_{0},T_{0}+1]\atop x\in[0,L]}\lambda^{-2}_{1}\big(\frac{\gamma+1+|\gamma-3|}{2}|m_{t}|+\frac{C_{3}}{2}e^{C_{3}(T_{0}+1)}\underline{u}\big)\notag\\
&\leq C^{2}_{2}\big(\frac{\gamma+1+|\gamma-3|}{2}K_{2}\sup\limits_{x\in[0,L]}\|m\|_{H^{2}([T_{0},T_{0}+1])}+\frac{C_{3}}{2}e^{C_{3}(T_{0}+1)}\underline{u}\big)\notag\\
&\leq C_{4}:=C^{2}_{2}\big(\frac{\gamma+1+|\gamma-3|}{2}K_{2}\delta+\frac{C_{3}}{2}e^{C_{3}(T_{0}+1)}\underline{u}\big),\label{4Pro1}
\end{align}
where $C_{3}=\max\limits_{t\geq0}\{|\alpha(t)|,|\alpha'(t)|,|\alpha''(t)|\}$, $K_{2}>0$ are constants. Denote $C_{5}=C_{4}+C_{2}C_{3}$, then \eqref{2pro1} turns into
\begin{equation}
\frac{d}{dx}\frac{1}{2}\|m\|_{L^{2}}^{2}\leq \frac{1}{2}C_{5}\|m\|_{L^{2}}^{2}.\label{2Pro1}
\end{equation}
Taking the derivative $\partial_{t}$ on both sides of~\eqref{c2}, multiplying it by $m_{t}^{T}$ and integrating it from $T_{0}$ to $T_{0}+1$, it holds that
\begin{align}
\frac{d}{dx}\frac{1}{2}\|m_{t}\|_{L^{2}}^{2}=&-\int_{T_{0}}^{T_{0}+1}m_{t}^{T}\Lambda(m_{t})_{t}dt+\int_{T_{0}}^{T_{0}+1}-m_{t}^{T}\Lambda_{t}m_{t}+\frac{\alpha(t)}{2}m_{t}^{T}\tilde{\Lambda}m_{t}dt\notag\\
&+\int_{T_{0}}^{T_{0}+1}\frac{\alpha'(t)}{2}m_{t}^{T}\tilde{\Lambda}m+\frac{\alpha(t)}{2}m_{t}^{T}\tilde{\Lambda}_{t}mdt\notag\\
=&\int_{T_{0}}^{T_{0}+1}\frac{1}{2}m_{t}^{T}\Lambda_{t}m_{t}dt+\int_{T_{0}}^{T_{0}+1}-m_{t}^{T}\Lambda_{t}m_{t}+\frac{\alpha(t)}{2}m_{t}^{T}\tilde{\Lambda}m_{t}dt\notag\\
&+\int_{T_{0}}^{T_{0}+1}\frac{\alpha'(t)}{2}m_{t}^{T}\tilde{\Lambda}m+\frac{\alpha(t)}{2}m_{t}^{T}\tilde{\Lambda}_{t}mdt\notag\\
=&\int_{T_{0}}^{T_{0}+1}-\frac{1}{2}m_{t}^{T}\Lambda_{t}m_{t}+\frac{\alpha(t)}{2}m_{t}^{T}\tilde{\Lambda}m_{t}dt\notag\\
&+\int_{T_{0}}^{T_{0}+1}+\frac{\alpha'(t)}{2}m_{t}^{T}\tilde{\Lambda}m+\frac{\alpha(t)}{2}m_{t}^{T}\tilde{\Lambda}_{t}mdt\notag\\
\leq&(\frac{1}{2}C_{4}+\frac{1}{2}C_{2}C_{3})\|m_{t}\|_{L^{2}}^{2}+\frac{1}{2}(C_{2}C_{3}+C_{3}C_{4})\int_{T_{0}}^{T_{0}+1}|m_{t}||m|dt\notag\\
\leq&\frac{1}{2}C_{6}\|m_{t}\|_{L^{2}}^{2}+\frac{1}{2}C_{7}\|m\|_{L^{2}}^{2},\label{5Pro1}
\end{align}
where we have used the integration by parts in the second identity, \eqref{3Pro1}-\eqref{4Pro1}, and the fact $\sup\limits_{t\in[T_{0},T_{0}+1]\atop x\in[0,L]}|\Lambda_{t}|=\sup\limits_{t\in[T_{0},T_{0}+1]\atop x\in[0,L]}|\tilde{\Lambda}_{t}|$.

Similarly, taking the derivative $\partial_{t}^{2}$ on both sides of~\eqref{c2}, multiplying it by $m_{tt}^{T}$ and integrating it from $T_{0}$ to $T_{0}+1$, we have
\begin{align}
\frac{d}{dx}\frac{1}{2}\|m_{tt}\|_{L^{2}}^{2}=&-\int_{T_{0}}^{T_{0}+1}m_{tt}^{T}\Lambda(m_{tt})_{t}dt+\int_{T_{0}}^{T_{0}+1}(-2m_{tt}^{T}\Lambda_{t}m_{tt}+\frac{\alpha(t)}{2}m_{tt}^{T}\tilde{\Lambda}m_{tt}\notag\\
&+\alpha(t) m_{tt}^{T}\tilde{\Lambda}_{t}m_{t}+\alpha'(t) m_{tt}^{T}\tilde{\Lambda}m_{t}+\frac{\alpha''(t)}{2}m_{tt}^{T}\tilde{\Lambda}m+\alpha'(t)m_{tt}^{T}\tilde{\Lambda}_{t}m\notag\\
&-m_{tt}^{T}\Lambda_{tt}m_{t}+\frac{\alpha(t)}{2}m_{tt}^{T}\tilde{\Lambda}_{tt}m)dt\notag\\
=&\int_{T_{0}}^{T_{0}+1}(-\frac{3}{2}m_{tt}^{T}\Lambda_{t}m_{tt}+\frac{\alpha(t)}{2}m_{tt}^{T}\tilde{\Lambda}m_{tt}+\alpha(t) m_{tt}^{T}\tilde{\Lambda}_{t}m_{t}\notag\\
&+\alpha'(t) m_{tt}^{T}\tilde{\Lambda}m_{t}+\frac{\alpha''(t)}{2}m_{tt}^{T}\tilde{\Lambda}m+\alpha'(t)m_{tt}^{T}\tilde{\Lambda}_{t}m\notag\\
&-m_{tt}^{T}\Lambda_{tt}m_{t}+\frac{\alpha(t)}{2}m_{tt}^{T}\tilde{\Lambda}_{tt}m)dt\notag\\
\leq&(\frac{3}{2}C_{4}+\frac{5}{4}C_{2}C_{3}+C_{3}C_{4})\|m_{tt}\|_{L^{2}}^{2}+\frac{1}{2}(C_{3}C_{4}+C_{2}C_{3})\| m_{t}\|_{L^{2}}^{2}\notag\\
&+(\frac{1}{2}C_{3}C_{4}+\frac{1}{4}C_{2}C_{3})\| m\|_{L^{2}}^{2}+\int_{T_{0}}^{T_{0}+1}(|\frac{\alpha(t)}{2}m_{tt}^{T}\tilde{\Lambda}_{tt}m|\notag\\
&+|-m_{tt}^{T}\Lambda_{tt}m_{t}|)dt. \label{6Pro1}
\end{align}

Since $\tilde{\Lambda}_{tt}=\tilde{\Lambda}_{1_{tt}}+\tilde{\Lambda}_{2_{tt}}^{1}+\tilde{\Lambda}_{2_{tt}}^{2}$ with
 \begin{align*}
\tilde{\Lambda}_{1_{tt}}&=\left(
            \begin{array}{cc}
              2\lambda_{2}^{-3}(\lambda_{2_t})^{2} & 2\lambda_{2}^{-3}(\lambda_{2_t})^{2}\\
              \\
               2\lambda_{1}^{-3}(\lambda_{1_t})^{2} & 2\lambda_{1}^{-3}(\lambda_{1_t})^{2}
            \end{array}
          \right),\\
          \\
\tilde{\Lambda}_{2_{tt}}^{1}&=\left(
            \begin{array}{cc}
              -\lambda_{2}^{-2}\frac{\alpha^{2}(t)}{2}e^{\int_{0}^{t}\alpha(s)ds}\underline{u} & -\lambda_{2}^{-2}\frac{\alpha^{2}(t)}{2}e^{\int_{0}^{t}\alpha(s)ds}\underline{u}\\
              \\
              -\lambda_{1}^{-2}\frac{\alpha^{2}(t)}{2}e^{\int_{0}^{t}\alpha(s)ds}\underline{u} & -\lambda_{1}^{-2}\frac{\alpha^{2}(t)}{2}e^{\int_{0}^{t}\alpha(s)ds}\underline{u}
            \end{array}
          \right),\\
          \\
\tilde{\Lambda}_{2_{tt}}^{2}&=\left(
            \begin{array}{cc}
              -\lambda_{2}^{-2}(\frac{3-\gamma}{2}m_{2_{tt}}+\frac{\gamma+1}{2}m_{1_{tt}}) & -\lambda_{2}^{-2}(\frac{3-\gamma}{2}m_{2_{tt}}+\frac{\gamma+1}{2}m_{1_{tt}})\\
              \\
              -\lambda_{1}^{-2}(\frac{\gamma+1}{2}m_{2_{tt}}-\frac{\gamma-3}{2}m_{1_{tt}}) & -\lambda_{1}^{-2}(\frac{\gamma+1}{2}m_{2_{tt}}-\frac{\gamma-3}{2}m_{1_{tt}})
            \end{array}
          \right),
\end{align*}
\vspace{0cm}
we have from~\lemref{t1}, \eqref{emd1} and~\eqref{4Pro1} that
\begin{align}
\int_{T_{0}}^{T_{0}+1}|\frac{\alpha(t)}{2}m_{tt}^{T}\tilde{\Lambda}_{tt}m|dt\leq&\frac{1}{2}C_{3}\int_{T_{0}}^{T_{0}+1}|m_{tt}||\tilde{\Lambda}_{tt}||m|dt\notag\\
\leq&\frac{1}{2}C_{3}\int_{T_{0}}^{T_{0}+1}|m_{tt}||\tilde{\Lambda}_{1_{tt}}||m|+|m_{tt}||\tilde{\Lambda}_{2_{tt}}^{1}|| m| dt\notag\\
&+\frac{1}{2}C_{3}\int_{T_{0}}^{T_{0}+1}|m_{tt}||\tilde{\Lambda}_{2_{tt}}^{2}||m|dt\notag\\
\leq&\frac{1}{2}C_{3}\sup_{t\in[T_{0},T_{0}+1]\atop x\in[0,L]}(|\tilde{\Lambda}_{1_{tt}}|+|\tilde{\Lambda}_{2_{tt}}^{1}|)\int_{T_{0}}^{T_{0}+1}|m_{tt}||m|dt\notag\\
&+\frac{1}{2}C_{3}\sup_{t\in[T_{0},T_{0}+1]\atop x\in[0,L]}|m|\int_{T_{0}}^{T_{0}+1}|m_{tt}||\tilde{\Lambda}_{2_{tt}}^{2}|dt\notag\\
\leq&\frac{1}{2}C_{3}C_{8}\int_{T_{0}}^{T_{0}+1}|m_{tt}||m|dt+\frac{1}{2}C_{3}C_{9}K_{1}\delta\int_{T_{0}}^{T_{0}+1}|m_{tt}|^{2} dt\notag\\
\leq&\frac{1}{4}C_{3}C_{8}(\int_{T_{0}}^{T_{0}+1}|m_{tt}|^{2} dt+\int_{T_{0}}^{T_{0}+1}|m|^{2} dt)\notag\\
&+\frac{1}{2}C_{3}C_{9}K_{1}\delta\int_{T_{0}}^{T_{0}+1}|m_{tt}|^{2} dt\notag\\
=&(\frac{1}{4}C_{3}C_{8}+\frac{1}{2}C_{3}C_{9}K_{1}\delta )\|m_{tt}\|_{L^{2}}^{2}+\frac{1}{4}C_{3}C_{8}\|m\|_{L^{2}}^{2},\label{c26}
\end{align}
where
\begin{align}
C_{8}&=16(\underline{u}-\underline{c})^{-3}C_{4}^{2}+2C^{2}_{3}(\underline{u}-\underline{c})^{-2}e^{C_{3}(T_{0}+1)}\underline{u},\label{1c26}\\
C_{9}&=4(\frac{\gamma+1}{2}+_{}\frac{\mid\gamma-3\mid}{2})(\underline{u}-\underline{c})^{-2}. \label{2c26}
\end{align}
By the similar arguments as above, we can get the estimate of $\int_{T_{0}}^{T_{0}+1}|-m_{tt}^{T}\Lambda_{tt}m_{t}|dt$ as follows
\begin{equation}\label{c27}
\begin{split}
\int_{T_{0}}^{T_{0}+1}|m_{tt}^{T}\Lambda_{tt}m_{t}|dt\leq (\frac{1}{2}C_{8}+C_{9}K_{2}\delta)\|m_{tt}\|_{L^{2}}^{2}+\frac{1}{2}C_{8}\|m_{t}\|_{L^{2}}^{2},
\end{split}
\end{equation}
where $C_{8}$ and $C_{9}$ are defined by~\eqref{1c26} and~\eqref{2c26} respectively. The details are omitted here.

Substituting~\eqref{c26} and~\eqref{c27} into~\eqref{6Pro1}, we have
\begin{equation}\label{c28}
\frac{d}{dx}\frac{1}{2}\|m_{tt}\|_{L^{2}}^{2}\leq\frac{1}{2}C_{10}\|m_{tt}\|_{L^{2}}^{2}+\frac{1}{2}C_{11}\|m_{t}\|_{L^{2}}^{2}+\frac{1}{2}C_{12}\| m\|_{L^{2}}^{2},
\end{equation}
where
\begin{align*}
C_{10}&=(\frac{1}{2}C_{3}C_{8}+C_{3}C_{9}K_{1}\delta+C_{8}+2C_{9}K_{2}\delta+3C_{4}+\frac{5}{2}C_{2}C_{3}+2C_{3}C_{4}),\\
C_{11}&=(C_{8}+C_{2}C_{3}+C_{3}C_{4}) ,\quad C_{12}=(\frac{1}{2}C_{3}C_{8}+\frac{1}{2}C_{2}C_{3}+C_{3}C_{4}).
\end{align*}
Combining~\eqref{2Pro1}, \eqref{5Pro1} and~\eqref{c28}, it holds that
\begin{equation*}
\begin{aligned}
&\frac{d}{dx}(\|m\|_{L^{2}}^{2}+\|m_{t}\|_{L^{2}}^{2}+\|m_{tt}\|_{L^{2}}^{2})\\
&\leq(C_{5}+C_{7}+C_{12})\|m\|_{L^{2}}^{2}+(C_{6}+C_{11})\|m_{t}\|_{L^{2}}^{2}+C_{10}\|m_{tt}\|_{L^{2}}^{2}\\
&\leq C_{13}(\|m\|_{L^{2}}^{2}+\|m_{t}\|_{L^{2}}^{2}+\|m_{tt}\|_{L^{2}}^{2}),
\end{aligned}
\end{equation*}
where $C_{13}=\max\{C_{5}+C_{7}+C_{12}, C_{6}+C_{11}, C_{10}\}$. Thus, we get~\eqref{1Pro1}. The proof of~\proref{Pro1} is completed.
\end{proof}
\begin{theorem} \label{thm1}
There are positive constants $\varepsilon_{0}$ and $C_{0}$, such that if
\begin{align*}
\|m_{0}\|_{H^{2}([T_{0},T_{0}+1])}\leq \varepsilon
\end{align*}
for any $\varepsilon\in(0,\varepsilon_{0})$, then there is a unique time-periodic solution $m\in C([0,L]; H^{2}\\([T_{0},T_{0}+1]))\cap C^{1}([0,L]; H^{1}([T_{0},T_{0}+1]))$ to the Cauchy problem~\eqref{c1}, which satisfies
\begin{align}
\sup\limits_{x\in[0,L]}\|m\|_{H^{2}([T_{0},T_{0}+1])}+\sup\limits_{x\in[0,L]}\|\partial_{x}m\|_{H^{1}([T_{0},T_{0}+1])}\leq C_{0}\varepsilon. \label{1thm1}
\end{align}
\end{theorem}
\begin{proof}
We can refer to~\cite{Majda} for the local existence and uniqueness of the $C^{1}$ solution to the Cauchy problem~\eqref{c1}.

Applying the Gronwall's inequality to~\eqref{1Pro1}, we have
\begin{align}
\sup\limits_{x\in[0,L]}\|m(t,x)\|^{2}_{H^{2}([T_{0},T_{0}+1])}\leq \|m_{0}\|^{2}_{H^{2}([T_{0},T_{0}+1])}e^{CL}. \label{2thm1}
\end{align}
By choosing $\varepsilon_{0}=e^{-\frac{CL}{2}}\delta$, if $\|m_{0}\|_{H^{2}([T_{0},T_{0}+1])}\leq \varepsilon$ for any $\varepsilon\in(0,\varepsilon_{0})$, we have from~\eqref{2thm1} that
\begin{align}
\sup\limits_{x\in[0,L]}\| m \|_{H^{2}([T_{0},T_{0}+1])}<\delta, \label{3thm1}
\end{align}
which verifies~\eqref{c3}. With the help of~\eqref{1Pro1} and~\eqref{3thm1}, we can easily check that
\begin{align}
m\in C([0,L]; H^{2}([T_{0},T_{0}+1])). \label{4thm1}
\end{align}
Furthermore, from~\eqref{c2}, \lemref{t1} and~\eqref{4thm1}, we can get $m_{x}\in C([0,L]; H^{1}([T_{0}\\,T_{0}+1]))$ and
\begin{align}
\sup\limits_{x\in[0,L]}\|m_{x}\|_{H^{1}([T_{0},T_{0}+1])}\leq C\varepsilon \label{5thm1}
\end{align}
for some constant $C>0$.

By the aid of the uniform a-priori estimates~\eqref{3thm1} and~\eqref{5thm1} and the standard continuity arguments, we can extend the local solution in $x\in[0,L]$ to obtain the time-periodic solution to the Cauchy problem~\eqref{c1}, which belongs to $C([0,L]; H^{2}([T_{0},T_{0}+1]))\cap C^{1}([0,L]; H^{1}([T_{0},T_{0}+1]))$ and satisfies~\eqref{1thm1}. The proof of Theorem \ref{thm1} is completed.
\end{proof}
Let $T_{1}>T_{0}$. We next consider the following initial boundary value problem for $m(t,x)=(m_{1}(t,x),m_{2}(t,x))^{T}:=(r(t,x)-r_{\alpha}(t), s(t,x)-s_{\alpha}(t))^{T}$ with $(t,x)\in[0,T_{1}]\times[0,L]$
\begin{equation}\label{PIB1}
\left\{\begin{aligned}
&m_{x}+\Lambda m_{t}=\frac{\alpha(t)}{2}\Lambda
\left(
\begin{array}{c}
 r-r_{\alpha}+s-s_{\alpha}\\
 r-r_{\alpha}+s-s_{\alpha}\\
 \end{array}
 \right),\\
&m(t,0)=m_{0}(t)=
(r_{l}(t)-r_{\alpha}(t), s_{l}(t)-s_{\alpha}(t))^{T},\\
&m(0,x)=0.
\end{aligned}
\right.
\end{equation}
For the smooth solution to the system~\eqref{PIB1}, it holds that
\begin{align*}
m(0,x)=m_{t}(0,x)=m_{tt}(0,x),\quad x\in[0,L].
\end{align*}
Noticing the positive definite of $\Lambda$
and applying the similar arguments to those in the poof of~\theref{thm1}, we have the following result.
\begin{theorem}\label{thm2}
There exist positive constants $\varepsilon_{1}$, $C_{1}$ and $T_{1}>T_{0}$, such that if
\begin{align*}
\|m_{0}\|_{H^{2}([0,T_{1}])}\leq \varepsilon
\end{align*}
for any $\varepsilon\in(0,\varepsilon_{1})$, then there is a unique solution $m\in C([0,L]; H^{2}([0,T_{1}]))\cap C^{1}([0,L]; H^{1}([0,T_{1}]))$ to the initial boundary value problem~\eqref{PIB1}, which satisfies
\begin{align}
\sup\limits_{x\in[0,L]}\|m\|_{H^{2}([0,T_{1}])}+\sup\limits_{x\in[0,L]}\|m_{x}\|_{H^{1}([0,T_{1}])}\leq C_{1}\varepsilon. \label{1thm2}
\end{align}
\end{theorem}
Now, it is time for us to give the proof of Theorem 1.1:

\textbf{Proof of~\theref{T1}}\quad Making an extension in time with a period $P$ of the solution given in~\theref{thm1} from $[T_{0},T_{0}+P]$ to $[T_{0},\infty)$, and putting together with it and the solution proved in~\theref{thm2}, we obtain a unique solution $m$ to the Cauchy problem~\eqref{DU1}, which satisfies~\eqref{PC1}.\\
\indent By the regularities of the solutions stated in~\theref{thm1} and~\theref{thm2} and applying the Sobolev embeddings $H^{1}(I)\hookrightarrow C(I)$ and $H^{2}(I)\hookrightarrow C^{1}(I)$ with $I=[T_{0},T_{0}+P]$ or $[0,T_{1}]$, we have $m\in C^{1}([0,\infty)\times[0,L])$. Furthermore, from~\eqref{1thm1} and~\eqref{1thm2}, we can obtain~\eqref{a5}. The proof of~\theref{T1} is completed.

\textbf{acknowledgements}\quad The authors would like to thank the referees for their valuable suggestions and comments. 

\textbf{Conflict of Interest}\quad The authors declared that they have no conflict of interest.

\end{document}